\documentclass{birkjour}

\usepackage[english]{babel}

 \newtheorem{thm}{Theorem}[section]
 
 \newtheorem{lem}[thm]{Lemma}
 
 \theoremstyle{definition}
 
 \theoremstyle{remark}
 \newtheorem*{remark}{Remark}
 
 \numberwithin{equation}{section}
 
 \newcommand{\RM}{\mathbb{R}}
 \newcommand{\bv}{\mathbf{v}}
 \newcommand{\bx}{\mathbf{x}}

 \newcommand{\bn}{\mathbf{n}}
 
 \newcommand{\ZM}{\mathbb{Z}}
 \newcommand{\CM}{\mathbb{C}}

 \newcommand{\curl}{\operatorname{curl}}
 \newcommand{\divv}{\operatorname{div}}
 \newcommand{\real}{\operatorname{Re}}
 \newcommand{\imag}{\operatorname{Im}}
 \newcommand{\ds}{\operatorname{ds}}
 
 \newcommand{\dlambda}{\operatorname{d\lambda}} 
 
 \newcommand{\sign}{\operatorname{sign}}
 \newcommand{\dx}{\operatorname{d\mathbf{x}}}

\begin{document}

\title[]
{2D Navier-Stokes system and no-slip boundary condition for vorticity in exterior domains}

\author[A.\,V.~Gorshkov]{A.\,V.~Gorshkov}

\address{%
Lomonosov Moscow State University,\\ 
Leninskie Gory, Moscow, 119991, \\Russian Federation}

 
\email{alexey.gorshkov.msu@gmail.com}


\date{January 1, 2004}

\begin{abstract}
In this article we derive Biot-Savar law for exterior 2D domain. No-slip condition in terms of vorticity 
is expressed in integral form. For non-stationary fluid dynamics equations no-slip condition generates
affine invariant manifold and no-slip integral relations on vorticity can be transferred to
Robin-type boundary condition. 
\end{abstract}

\maketitle

\tableofcontents

\section*{Preliminary} Many papers as well as current one are devoted to the question - does there exist any boundary condition for vorticity which corresponds to no-sleep ones - zero velocity on the boundary of the solid. In a playful way we can say that dolphins know the answer on this question, but keep silence. Dolphin's skin can read vorticity distribution in order to prevent turbulent flow. 

But of course the answer on the above question is negative - there is no any analogous to no-slip condition only in terms of vorticity without additional functions involved in (e.g. stream function, velocity). Biot-Savar law restores solenoidal velocity field $\bv(\bx)$ induced by vorticity $w(\bx)$. It expresses vorticity via velocity field by some integral relations. If the flow interacts with solid by no-slip condition, then it turns to zero integral relations in Biot-Savar law which do not admit explicit integration. It seems to negate all efforts to find boundary vorticity. Nevertheless, the situation isn't so definite.

Why it is so important to derive vorticity boundary condition? Fluid dynamics equations written in terms of vorticity favorably differ from ones covering velocity evolution. So, for 2D and 3D Navier-Stokes system the vorticity dynamics involves fewer number of equations and fewer number of unknown functions. If the system doesn't involve boundary condition (e.g. Cauchy problem), then $\curl$ operator significantly simplifies corresponding {{\it initial-value problem}. But in case of {\it initial-boundary-value problem} it isn't true by the reason above - Dirichlet no-slip boundary turn out to sophisticated integral condition.  

Consider initial-boundary-value problem for the nonstationary Navier-Stokes system defined in exterior domain $\Omega$ modelling flow around solid with given constant horizontal flow at infinity $\bv_\infty = (\bv_{1,\infty},\bv_{2,\infty}) \in \RM^2$:
\begin{eqnarray}
&&\partial_t \bv - \Delta \bv +(\bv,\nabla)\bv = 0= \nabla p \label{maineqns}\\
&&{\rm div}~\bv(t,\bx)=0 \label{freedivns}\\
&&\bv(0,\bx)=\bv_0(\bx) \label{initns}\\
&&\bv(t,\bx)=0,~|\bx|=r_0 \label{boundns}\\
&&\bv(t,\bx) \to \bv_\infty,~|\bx|\to \infty. \label{boundinfns}
\end{eqnarray}
Here $\bv(t,\bx)=(v_1(t,\bx),v_2(t,\bx))$ is the velocity field and $p(t,\bx)$ is the pressure.

Applying the curl operator  
$w(t,\bx)=$ ${\rm curl}~\bv(t,\bx)$ $=\partial_{\bx_1}v_2 -  \partial_{\bx_2}v_1$ we get the system of equations for vorticity
\begin{eqnarray} 
\frac{\partial w(t,\bx)}{\partial t}	 - \Delta w + (\bv,\nabla)w  = 0,  \label{maineqw} \\
w(0,\bx)=w_0(\bx) \label{initw}\\
\curl^{-1} w(t,\bx) \Big|_{|\bx|=r_0} = 0, \label{boundw}\\
w(t,\bx) \to 0,~|\bx|\to \infty \label{boundinfw}
\end{eqnarray} 
with initial datum $w_0(\bx)={\rm curl}~\bv_0(\bx)$. 

The inverse operator $\curl^{-1}$ arises a lot of questions like existence and uniqueness. Exterior domain with one additional boundary point at infinity only complicates this situation. We will find sufficient condition for well-posedness of this operator.

For 2D Navier-Stokes system and its approximations - Stokes and Oseen equations - we will derive boundary condition for vorticity. It will be expressed in terms of Fourier coefficients, so it will implicitly involve angular integration, but nevertheless, it will be pure initial-boundary-value problem. 

In the paper we will exploit $\RM^2$ as real as well as complex plane depending on the context.
For points from $\RM^2$ we will use different real and complex notations including polar coordinates $r,\varphi$ such as 
$\bx=(x_1,x_2)$, $z=x_1+ix_2$, $z=r e^{i\varphi}$	. 

\section{Biot-Savar law in exterior domains}

Now we study when the solenoidal velocity field $\bv(\bx)$ can be uniquely restored from its vorticity $w(\bx)$. Consider the following elliptic problem corresponding to Biot-Savar law in exterior domain $\Omega$:
\begin{eqnarray}
&&\rm{div}~ \bv(\bx) = 0, \label{freediv} \\
&&\rm{curl}~ \bv(\bx) = w(\bx), \label{curleq} \\
&&\bv(\bx)=0,~\bx \in \partial \Omega, \label{bound}\\
&&\bv(\bx)\to\bv_\infty,~|\bx|\to \infty, \label{boundinf}.
\end{eqnarray}

Exterior domains are not simply connected and the problem above could haven't uniquie solution. For example, equations (\ref{freediv}), (\ref{curleq}) supplied with slip condition on the boundary 
\begin{equation}\label{slip}
\left(\bv(\bx), \bn \right) = 0,~\bx \in \partial \Omega,
\end{equation}
and fixed flow at infinity (\ref{boundinf}) have unique solution only if we fix circularity at infinity ($\bn$ is an outer normal to boundary). No-slip condition (\ref{bound}) is stronger than (\ref{slip}), and so some additional restrictions on $w(\bx)$ are required. These restrictions can be realized via moment relations for vorticity. In \cite{AG} the solvability of the system above was researched in details for slip and no-slip conditions in exterior of the disc. Here we extend these results on more general domains and obtain boundary analogue of no-slip condition.

So, we need to fix circularity at infinity. From physical point of view it is natural to suppose zero-circularity:
\begin{equation} \label{zerocirculation}
\lim_{R\to\infty}\oint_{|\bx|=R} \bv \cdot d\mathbf{l} = 0.
\end{equation} 

The relationship between Cartesian and polar coordinate systems for $\bv_\infty$ is given by formulas:
\begin{align*}
v_{\infty,r}=v_{\infty,x}\cos \varphi + v_{\infty,y}\sin \varphi \\ 
v_{\infty,\phi}=v_{\infty,y}\cos \varphi - v_{\infty,x}\sin \varphi. 
\end{align*}

Then its Fourier coefficients are determined as
\begin{align}
v_{\infty,r,k}=\frac{\delta_{|k|,1}}2 (v_{\infty,x} - i k v_{\infty,y}) \label{fouriercoeffr} \\ 
v_{\infty,\phi,k}=\frac{\delta_{|k|,1}}2 (v_{\infty,y} + i k v_{\infty,x}) \label{fouriercoeffphi},
\end{align} 
and
\begin{align} \label{vrvphi}
v_{\infty,\phi,k} = \left\{
{\begin{matrix}  iv_{\infty,r,k}, k \geq 0, \\
-iv_{\infty,r,k}, k<0.
\end{matrix}}\right.	
\end{align}

All Fourier coefficients of external flow equal to zero except $k=\pm 1$. For horizontal flow $\bv_\infty=(v_\infty,0)$
\begin{align*}
v_{\infty,r,k}=\frac{\delta_{|k|,1}}2 v_\infty \\ 
v_{\infty,\phi,k}=i k\frac{\delta_{|k|,1}}2  v_\infty.
\end{align*}

We have standard equities
$$
v_{\infty,r,k} = \overline{v_{\infty,r,-k}},  v_{\infty,\phi,k} = \overline{v_{\infty,\phi,-k}}. 
$$

\section{Boundary condition in exterior of the disc}
In this section the domain under investigation will be exterior of the disc $B_{r_0}=\{\bx \in \RM^2,~|\bx| > r_0 \},~r_0>0$. We will derive Biot-Savar law and no-slip boundary condition for Stokes and Navier-Stokes systems. 

\subsection{Biot-Savar law}
In polar coordinates equations (\ref{freediv}),(\ref{curleq}) can be written in Fourier coefficients $v_{r,k}$, $v_{\varphi,k}$:
\begin{eqnarray*}
&&{\frac {1}{r}}{\frac {\partial }{\partial r}}\left(rv_{r,k}\right)+{\frac {ik}{r}} v_{\varphi,k} = 0,\\
&&{\frac {1}{r}}{\frac {\partial }{\partial r}}\left(rv_{\varphi,k}\right)-{\frac {ik}{r}} v_{r,k} = w_k.
\end{eqnarray*}

The basis for solutions of homogeneous system when $w_k \equiv 0$ consists of two vectors:
\begin{align*}
\begin{pmatrix}
v^1_{r,k} \\
v^1_{\varphi,k} 
\end{pmatrix} 
=
\begin{pmatrix}
ir^{-k-1} \\
r^{-k-1} 
\end{pmatrix}
, \\
\begin{pmatrix}
v^2_{r,k} \\
v^2_{\varphi,k} 
\end{pmatrix} 
=
\begin{pmatrix}
ir^{k-1} \\
-r^{k-1} 
\end{pmatrix}.
\end{align*}

The solution of this system with boundary relations (\ref{bound}), (\ref{boundinf}) and zero-circularity condition (\ref{zerocirculation}) was derived in \cite{AG} as Biot-Savar law in exterior of the disc in the following form for $k \in \ZM$:
\begin{eqnarray} \label{BiotSavar1}
&&v_{r,k} = \sign(k)  \frac{ir^{-|k|-1}}2 \int_{r_0}^r s^{|k|+1}w_k(s)\ds \\ 
&&~~~~~~~~~~~~~~~~~~+ \sign(k) \frac{ir^{|k|-1}}2 \int_r^\infty s^{-|k|+1}w_k(s)\ds  + v_{\infty,r,k} \nonumber \\ \label{BiotSavar2}
&&v_{\varphi,k} = \frac{r^{-|k|-1}}2 \int_{r_0}^r s^{|k|+1}w_k(s)\ds \nonumber \\ 
&&~~~~~~~~~~~~~~~~~~- \frac{r^{|k|-1}}2 \int_r^\infty s^{-|k|+1}w_k(s)\ds + v_{\infty,\phi,k}.
\end{eqnarray} 

No-slip condition (\ref{bound}) and (\ref{vrvphi}) lead to moment relations on vorticity ($k\in \ZM$): 
\begin{equation} \label{noslipcondintegral}
\int_{r_0}^\infty s^{-|k|+1}w_k(s)\ds = 2ik v_{\infty,r,k} = 2 v_{\infty,\phi,k}. 
\end{equation}
From (\ref{fouriercoeffr}), (\ref{fouriercoeffphi}) these moments don't equal to zero only if $|k|=1$. 

We rewrite moment relationship in terms of vorticity $w(\bx)$ for $k\geq 0$:
\begin{align*}
\int_{r_0}^\infty s^{-|k|+1}w_k(s)\ds = \frac 1{2 \pi} \int_{r_0}^\infty \int_0^{2\pi} s^{-k+1}w(s,\varphi)e^{-ik\varphi}\ds d\varphi \\ =\frac 1{2 \pi} \int_{B_{r_0}} \frac {w(\bx)}{z^k} \dx = 2ik v_{\infty,r,k},~k\geq 0.
\end{align*} 

For $k<0$ we have absolutely the same moments equity:
\begin{align*} 
\int_{r_0}^\infty s^{-|k|+1}w_k(s)\ds = \frac 1{2 \pi} \int_{r_0}^\infty \int_0^{2\pi} s^{k+1}w(s,\varphi)e^{-ik\varphi}\ds d\varphi \nonumber \\ =\frac 1{2 \pi} \int_{B_{r_0}} \overline{z^k} w(\bx) \dx = 2ik v_{\infty,r,k},~k<0.
\end{align*}
The last formula for $k<0$ is just complex conjugation of the analogous formula for $k \geq 0$.

Then  for no-slip condition we have affine subspace $M$ which must be invariant under Stokes, Oseen and Navier-Stokes flows:
\begin{align}\label{noslipcondintegral2} 
M=\left \{ w(\bx)\in L_1(B_{r_0})~\Big |~ 
\int_{B_{r_0}} \frac {w(\bx)}{z^k} \dx =  
4 \pi ik v_{\infty,r,k},~k \geq 0  \right \}.
\end{align} 
Note, that from (\ref{fouriercoeffr}), (\ref{fouriercoeffphi}) the integral relations in definition of $M$  are non-zero ones only if $k=1$. 

For Fourier coefficients the invariance of $M$ means that for vorticity flow which evolutes by cooefficients $w_k(t,\cdot)$ holds
$$
\int_{r_0}^\infty s^{-|k|+1}w_k(t,s)\ds = const,~k \in \ZM.
$$

In \cite{AG} was proved
\begin{thm}[Biot-Savart Law in polar coordinates]\label{Biotpolarnoslip}
If $w(\bx) \in M$ then there exists the unique solution of (\ref{freediv}) - (\ref{boundinf}), (\ref{zerocirculation}) given by (\ref{BiotSavar1}), (\ref{BiotSavar2}) 
with Fourier coefficients $v_{r,k}$, $v_{\varphi,k} \in L_\infty(r_0,\infty)$.
\end{thm}

Finally we have
\begin{thm} \label{invthm}
If initial datum $\bv_0(\bx)$, such that $\curl \bv_0\in L_1(B_{r_0})$ satisfies no-slip condition (\ref{bound}), infinity condition (\ref{boundinf}), zero-circularity (\ref{zerocirculation}), and $w(t,\cdot)=\curl \bv(t,\cdot)$  be the vorticity flow. Then $\curl \bv_0 \in M$ and in order to conserve no-slip condition $M$ must be invariant under the flow, e.g. for any time $t>0$ $w(t,\cdot) \in M$.
\end{thm}

\subsection{Stokes system} 

Consider Stokes flow for vorticity
\begin{align} \label{stokeseq}
\frac{\partial w(t,\bx)}{\partial t} - \Delta w(t,\bx)  = 0.
\end{align}
Supply it with Robin-type boundary condition:
\begin{equation}\label{robin_bound}
r_0\frac{\partial w_k(t,r)}{\partial r}\Big|_{r=r_0} + |k| w_k(t,r_0) = 0,~k \in \ZM
\end{equation}
and
\begin{equation}
w(t,\bx) \to 0,~|\bx|\to \infty. \label{boundinfvorticity}
\end{equation}

Then $M$ is invariant under the flow $w(t,\cdot)$. Indeed, fix $k>0$ and divide equation (\ref{stokeseq}) by $z^k$ and integrate over exterior of the disc $B_{r_0}$. From moment relations (\ref{noslipcondintegral2}) follows
$$
\frac d{dt}\int_{B_{r_0}} \frac {w(t,\bx)}{z^k}  \dx = 0
$$ 
and thus
$$
\int_{B_{r_0}}  \frac {\Delta w}{z^{k}}\dx=0.
$$
In other hand
\begin{align}\label{integrationbypart_dissip}\nonumber
\int_{B_{r_0}} \frac {\Delta w}{z^{k}}\dx = \int_{r_0}^\infty \int_0^{2\pi} \frac {\Delta w}{s^ke^{ik\phi}}sdsd\varphi = 2\pi \int_{r_0}^\infty s^{-k+1}\Delta_k w_k(s) ds \\ \nonumber=-2\pi \int_{r_0}^\infty s^{-|k|} \left (  \frac {\partial}{\partial s}\left(s \frac {\partial}{\partial s}w_k(t,s)\right) - \frac{k^2}{s}  w_k(t,s) \right ) \ds \\ \nonumber=
- r_0^{-|k|+1}\frac{\partial w_k(t,r)}{\partial r}\Big|_{r=r_0} + \int_{r_0}^\infty s^{-|k|} \left (  |k|  \frac {\partial}{\partial s}w_k(t,s) - \frac{k^2}{s} w_k(t,s) \right ) \ds \\  = -2\pi r_0^{-k}\left(r_0 \frac{\partial w_k(t,r)}{\partial r}\Big|_{r=r_0} + k w_k(t,r_0) \right )=0.
\end{align}

In similar way using complex conjugation of moment relations (\ref{noslipcondintegral2}) we obtain boundary condition for $k<0$:
\begin{align*}
\int_{B_{r_0}} {\Delta w}\overline{z^{k}}\dx  = \int_{r_0}^\infty \int_0^{2\pi} \frac {\Delta w}{s^{-k}e^{ik\phi}}sdsd\varphi = 2\pi \int_{r_0}^\infty s^{k+1}\Delta_k w_k(s) ds \\= -2\pi r_0^{k}\left(r_0 \frac{\partial w_k(t,r)}{\partial r}\Big|_{r=r_0} - k w_k(t,r_0) \right )=0.
\end{align*}

The boundary-value problem (\ref{stokeseq}), (\ref{robin_bound}), (\ref{boundinfvorticity})  can be solved using special {\it hydrodynamical Fourier transform} based on generalised Weber-Orr transform (see \cite{AG} for more details):
\begin{equation}\label{int:weberorr}
W_{|k|,|k|-1}[f](\lambda) = \int_{r_0}^\infty R_{|k|,|k|-1}(\lambda,s) f(s) s \ds,~k\in \ZM,
\end{equation}
where
$$
R_{k,l}(\lambda,s) = J_{k}(\lambda s)Y_{l}(\lambda r_0) - Y_{k}(\lambda s)J_{l}(\lambda r_0), 
$$
$J_{k}(r)$, $Y_{k}(r)$ - are the Bessel functions of the first and second type (see \cite{BE}).

\vskip 5pt
The inverse transform is defined by the formula
\begin{equation}\label{int:weberorrinv}
W^{-1}_{|k|,|k|-1}[\hat f](r) = \int_{0}^\infty \frac{R_{|k|,|k|-1}(\lambda,r)}{{J_{|k|-1}^2(\lambda r_0) + Y_{|k|-1}^2(\lambda r_0)}} \hat f (\lambda) \lambda \dlambda.
\end{equation}

In \cite{AG} was proved the following
\begin{thm} Let vector field $\bv_0(\bx)$ satisfies (\ref{freediv}), (\ref{bound}), (\ref{boundinf}), (\ref{zerocirculation}), $\curl \bv_0(\bx)$ $ \in L_1(B_{r_0})$, and its Fourier series as well as Fourier series for vorticity converges and coefficients $w_k^0(r)$ satisfy $w_k^0(r) \sqrt r \in L_1(r_0,\infty)$, $k \in \ZM$. Then $w(t,\bx)=\curl \bv (t, \bx)$ satisfies equation (\ref{stokeseq}), boundary conditions (\ref{robin_bound}),  (\ref{boundinfvorticity}) and is given via Fourier coefficients:
$$
w_k(t,r) = W^{-1}_{|k|,|k|-1} \left [ e^{-\lambda^2 t} W_{|k|,|k|-1} [w^0_k(\cdot)](\lambda) \right ](t,r),
$$
where $W_{|k|,|k|-1}$, $W^{-1}_{|k|,|k|-1}$ are the hydrodynamical Fourier transforms (\ref{int:weberorr}), (\ref{int:weberorrinv}).
\end{thm}

This formula gives explicit form of the solution for Stokes system in exterior of the disc in terms of vorticity. Using Biot-Savar law, one can get explicit formula for velocity. Note, that the inverse transform (\ref{int:weberorrinv}) is valid only with some restrictions on $\hat f (\lambda)$ which are fullfilled for $w^0_k(\cdot)$. The invertebility of the generalised Weber-Orr transform was studied in details in \cite{AG}.

\subsection{Oseen system}

Here we transfer integral relations (\ref{noslipcondintegral}) into boundary conditions for Oseen system. Navier-Stokes system for vorticity in polar coordinates is given by one equation
\begin{equation} \label{omegaeqpolar}
\frac{\partial w_k(t,r)}{\partial t}	 - \Delta_k w(t,r) + [(\bv,\nabla w)]_k = 0, ~w_k(0,r) = w^0_k(r),
\end{equation} 
where
$$
\Delta_k w_k(t,r) = \frac 1r \frac {\partial}{\partial r}\left(r \frac {\partial}{\partial r}w_k(t,r)\right) - \frac{k^2}{r^2} w_k(t,r),
$$
and $\bv$ is restored from $w$ via Biot-Savar law.

Consider Oseen approximation, where $\bv=\bv_\infty$ in (\ref{omegaeqpolar}). Evolution for Oseen approximation of Navier-Stokes flow around vector field $\bv_\infty$ is given by equation
\begin{align} \label{omegaeqoseenpolar}
\frac{\partial w_k(t,r)}{\partial t}	 - \Delta_k w(t,r) + [(\bv_\infty,\nabla w)]_k = 0, ~w_k(0,r) = w^0_k(r).
\end{align}

From Theorem \ref{invthm} the subspace $M$ must be invariant under the flow defined by (\ref{omegaeqpolar}) in order to satisfy boundary condition (\ref{bound}). 

Supply equation (\ref{omegaeqoseenpolar}) by boundary condition
\begin{equation}\label{robin_bound_oseen}
r_0 \frac{\partial w_k(t,r)}{\partial r}\Big|_{r=r_0} + |k| w_k(t,r_0)= \\
 r_0(v_{\infty,r,1} w_{k-1}(r_0) + v_{\infty,r,-1} w_{k+1}(r_0))
\end{equation}
and 
\begin{equation}
w(t,\bx) \to 0,~|\bx|\to \infty. \label{boundinfvorticityoseen}
\end{equation}

\begin{lem}The subspace $M$ stays invariant under the flow defined by (\ref{omegaeqoseenpolar}), (\ref{robin_bound_oseen}), (\ref{boundinfvorticityoseen}).
\end{lem}

\begin{proof} Fourier coefficients for the inertial term $(\bv_\infty,\nabla w)$ have the following form:
\begin{align*}
[(\bv_\infty,\nabla w)]_k=v_{\infty,r,1} w_{k-1}' + v_{\infty,r,-1} w_{k+1}'+\\
\frac ir
\left( (k-1)v_{\infty,\phi,1}w_{k-1} + (k+1)v_{\infty,\phi,-1}w_{k+1} \right).
\end{align*}

Multiply (\ref{omegaeqoseenpolar}) by $r^{-|k|+1}$ and integrate it from $r_0$ to $\infty$. First, calculate
\begin{align*}
\int_{r_0}^\infty s^{-|k|+1} [(\bv_\infty,\nabla w)]_k \ds 
=r_0^{-|k|+1}(v_{\infty,r,1} w_{k-1}(r_0) + v_{\infty,r,-1} w_{k+1}(r_0))\\ 
+(|k|-1)\int_{r_0}^\infty s^{-|k| } ( v_{\infty,r,1} w_{k-1} + v_{\infty,r,-1} w_{k+1}) \ds\\  
+i\int_{r_0}^\infty s^{-|k|} \left ((k-1)v_{\infty,\phi,1}w_{k-1} + (k+1)v_{\infty,\phi,-1}w_{k+1}\right ) \ds \\ 
=const(k) (|k|-1) \int_{r_0}^\infty s^{-|k| } w_{k+\sign k} \ds \\+ r_0^{-|k|+1}(v_{\infty,r,1} w_{k-1}(r_0) + v_{\infty,r,-1} w_{k+1}(r_0))
=r_0^{-|k|+1}(v_{\infty,r,1} w_{k-1}(r_0)\\ + v_{\infty,r,-1} w_{k+1}(r_0)).
\end{align*}
Here we used followed from (\ref{noslipcondintegral}) moments relation
$$
(|k|-1)\int_{r_0}^\infty s^{-|k| } w_{k+\sign k} \ds = 0.
$$
Then with help of (\ref{integrationbypart_dissip})
\begin{align*}
0=\frac d{dt}\int_{r_0}^\infty s^{-|k|+1}w_k(t,s)\ds \\= \int_{r_0}^\infty s^{-|k|+1} \Delta_k w_k(t,s)\ds -
\int_{r_0}^\infty s^{-|k|+1} [(\bv_\infty,\nabla w)]_k \ds \\
=  - r_0^{-|k|}\left(r_0 \frac{\partial w_k(t,r)}{\partial r}\Big|_{r=r_0} + |k| w_k(t,r_0) \right ) \\
+ r_0^{-|k|+1}(v_{\infty,r,1} w_{k-1}(r_0) + v_{\infty,r,-1} w_{k+1}(r_0)).
\end{align*}
 
\end{proof}

\begin{remark}
Right side of boundary condition (\ref{robin_bound_oseen}) arises as additional term in integration by parts of inertial part $(\bv_\infty,\nabla w)$. For Navier-Stokes system this term vanishes in $(\bv,\nabla w)$ due to no-slip condition. So, boundary (\ref{robin_bound}) stays actual and for Oseen approximation.
\end{remark}

\subsection{Navier-Stokes system}

Now we are ready to derive boundary condition to Navier-Stokes system. In fact it will be integral condition. But its first approximation will be the same boundary condition (\ref{robin_bound}) as for Stokes flow.

Consider Navier-Stokes equation for vorticity
\begin{align} \label{nssystem}
\frac{\partial w(t,\bx)}{\partial t}	 - \Delta w(t,\bx) + (\bv,\nabla w) = 0.
\end{align}

\begin{thm}Given initial datum $\bv_0(\bx)$, $\curl \bv_0\in L_1(B_{r_0})$ satisfying no-slip condition (\ref{bound}), infinity condition (\ref{boundinf}), zero-circularity (\ref{zerocirculation}) and $\bv(t,\bx)$ be the solution of (\ref{maineqns})-(\ref{boundinfns}) in $B_{r_0}$. Then $w(t,\bx)=\curl \bv(t,\bx)$ satisfies 
\begin{align}
&\left(r_0 \frac{\partial w_k(t,r)}{\partial r}\Big|_{r=r_0} +  |k| w_k(t,r_0) \right ) = \nonumber \\
&~~~~\left \{ {  \begin{matrix}
\frac {|k| r_0^{|k|}}{2\pi}
\int_{B_{r_0}} (\bv^\CM_\infty - \bv^\CM){z^{-|k|-1}} w(t,\bx)\dx,~k\geq 0, \\
\frac {|k| r_0^{|k|}}{2\pi}
\int_{B_{r_0}} \overline{(\bv^\CM_\infty - \bv^\CM)z^{-|k|-1}} w(t,\bx)\dx,~k<0.
\end{matrix} } \right . \label{robin_bound_ns}
\end{align}
\end{thm}

\begin{proof}
Fix $k>0$ and divide this equation by $z^k$ and integrate over exterior of the disc $B_{r_0}$. From moment relations (\ref{noslipcondintegral2}) follows
$$
\frac d{dt}\int_{B_{r_0}} \frac {w(t,\bx)}{z^k}  \dx = 0
$$

Denote $$\bv^\CM = v_1+iv_2,~\bv^\CM_\infty = v_{1,\infty}+iv_{2,\infty}.$$
 
Then
\begin{align*}
\int_{B_{r_0}} \frac {(\bv,\nabla w)}{z^k} \dx = -\int_{B_{r_0}} (\bv,\nabla z^{-k})w(t,\bx) \dx \\
=k \int_{B_{r_0}} \frac{v_1+iv_2}{z^{k+1}} w(t,\bx) \dx=k \int_{B_{r_0}} \frac{\bv^\CM}{z^{k+1}} w(t,\bx)\dx
\end{align*}

From (\ref{noslipcondintegral2})
$$
k \int_{B_{r_0}} \frac{\bv^\CM_\infty}{z^{k+1}} w(t,\bx)\dx = 0
$$
and
$$
\int_{B_{r_0}} \frac {(\bv,\nabla w)}{z^k} \dx = k \int_{B_{r_0}} \frac{\bv^\CM-\bv^\CM_\infty}{z^{k+1}} w(t,\bx)\dx
$$

Using (\ref{integrationbypart_dissip}) we get
$$
\int_{B_{r_0}}  \frac {\Delta w}{z^{k}}\dx=-2\pi r_0^{-k}\left(r_0 \frac{\partial w_k(t,r)}{\partial r}\Big|_{r=r_0} + k w_k(t,r_0) \right ).
$$

Then for $k\geq 0$
\begin{align*}
\left(r_0 \frac{\partial w_k(t,r)}{\partial r}\Big|_{r=r_0} + k w_k(t,r_0) \right ) = \frac {k r_0^k}{2\pi}
\int_{B_{r_0}} \frac{\bv^\CM_\infty - \bv^\CM}{z^{k+1}} w(t,\bx)\dx.
\end{align*}

Since $w_{-k}(t,r) = \overline{w_k(t,r)}$ then (\ref{robin_bound_ns}) holds for $k<0$. Theorem is proved.

\end{proof}

%
%

%
%

\begin{remark}
If $\|\bv - \bv_\infty\|$ is small in some integral norm, then right side in (\ref{robin_bound_ns}) transferes to boundary condition (\ref{robin_bound}). So, (\ref{robin_bound}) becomes rather accurate approximation for Navier-Stokes system. From this fact naturally occurs boundary control problem with unknown function $u_k(t)$:
\begin{equation}\label{robin_bound_control}
r_0\frac{\partial w_k(t,r)}{\partial r}\Big|_{r=r_0} + |k| w_k(t,r_0) = u_k(t),~k \in \ZM,
\end{equation}
where feedback control $u_k(t)$ is treated as small correction of zero boundary condition at each time step.
\end{remark}

\section{Boundary condition in exterior of simply connected domains}

In this section we derive Biot-Savar law for more general domains. We will establish crucial fact that Robin-type boundary (\ref{robin_bound}) stays actual for these domains.

\subsection{Biot-Savar Law}
Let $\Omega=\RM^2 \setminus B$, where $B$ is bounded simple-connected domain with piecewise smooth boundary and $\Phi$ be a Riemann mapping from $\Omega$ into exterior of the disc $B_{r_0}$ such that
\begin{align*}
\Phi(z)=z+O\left (\frac 1z \right ) \\
\Phi'(z)=1+O\left (\frac 1{z^2} \right ).
\end{align*}

Then $\bv=\bv(\Phi^{-1}(z))=\bv(\real \Phi^{-1}(x_1+ix_2), \imag \Phi^{-1}(x_1+ix_2))$ defines vector field in $B_{r_0}$. With help of Cauchy–Riemann relationship the equations (\ref{freediv}), (\ref{curleq}) turn to relation
\begin{equation}
\curl \bv + i \divv \bv = \frac {\Phi'(z)}{|\Phi'(z)|^2}w.
\end{equation}

Then system (\ref{freediv})-(\ref{boundinf}) after Riemann mapping in $B_{r_0}$ takes the form
\begin{eqnarray}
&&\rm{div}~ \bv(\bx) = \imag \overline {\Phi'^-1(z)} w(\bx)  \label{freediv3}\\
&&\rm{curl}~ \bv(\bx) = \real \overline {\Phi'^-1(z)} w(\bx)  \label{curleq3}\\
&&\bv(\bx)=0,~|\bx|=r_0  \label{bound3}\\
&&\bv(\bx)\to\bv_\infty,~|\bx|\to \infty.  \label{boundinf3}
\end{eqnarray}

Denote
\begin{align*}
r_k(r)=[\imag \overline {\Phi'^{-1}(z)} w(\bx)]_k,\\
q_k(r)=[\real \overline {\Phi'^{-1}(z)} w(\bx)]_k.
\end{align*}

Rewrite (\ref{freediv3}),(\ref{curleq3}) in polar coordinates in terms of Fourier coefficients $v_{r,k}$, $v_{\varphi,k}$:
\begin{eqnarray*}
&&{\frac {1}{r}}{\frac {\partial }{\partial r}}\left(rv_{r,k}\right)+{\frac {ik}{r}} v_{\varphi,k} = r_k(r),\\
&&{\frac {1}{r}}{\frac {\partial }{\partial r}}\left(rv_{\varphi,k}\right)-{\frac {ik}{r}} v_{r,k} = q_k(r).
\end{eqnarray*}

Under assumption of zero-circularity (\ref{zerocirculation}) from Stokes' Theorem we have 
\begin{align*}
\lim_{R\to\infty}\oint_{|\bx|=R} \bv(\bx) \cdot d\mathbf{l}= \frac 1{2 \pi} \int_{\Omega} w(\bx) \dx \\ = \frac 1{2 \pi} \int_{B_{r_0}} \frac {w(\bx)}{|\Phi'|^2} \dx = 
\int_{r_0}^\infty  \left [\frac {w(\bx)}{|\Phi'|^2} \right ]_{k=0} s \ds = 0.
\end{align*}

It guaranties solution's uniqueness of the above system. Existence will be provided by moment relations below. The solution of the above system for $k \in \ZM$ is derived by the same way as formulas (\ref{BiotSavar1}), (\ref{BiotSavar2}) in exterior of the disc (see \cite{AG} for more details):
\begin{eqnarray} \label{BiotSavar1_2}
&&v_{r,k} = \sign(k)  \frac{ir^{-|k|-1}}2 \int_{r_0}^r s^{|k|+1}(q_k-i\sign(k) r_k)\ds \\ 
&&~~~~~~~~~~~~~~~~~~+ \sign(k) \frac{ir^{|k|-1}}2 \int_r^\infty s^{-|k|+1}(q_k+i\sign(k)r_k)\ds  + v_{\infty,r,k} \nonumber \\ \label{BiotSavar2_2}
&&v_{\varphi,k} = \frac{r^{-|k|-1}}2 \int_{r_0}^r s^{|k|+1}(q_k-i\sign(k)r_k)\ds \nonumber \\ 
&&~~~~~~~~~~~~~~~~~~- \frac{r^{|k|-1}}2 \int_r^\infty s^{-|k|+1}(q_k+i\sign(k)r_k)\ds + v_{\infty,\phi,k}.
\end{eqnarray}

Formulas (\ref{BiotSavar1_2}), (\ref{BiotSavar2_2}) combined with (\ref{bound}) lead to relations on vorticity ($k\in \ZM$): 
\begin{equation} \label{noslipcondintegralrieman}
\int_{r_0}^\infty s^{-|k|+1}\left ( q_k(s)+i\sign(k)r_k(s) \right )\ds = 2ik v_{\infty,r,k} = 2 v_{\infty,\phi,k}. 
\end{equation}

This equities can be written in terms of $w(\bx)$. For $k\geq 0$ we have
\begin{align*} 
&\int_{r_0}^\infty s^{-|k|+1}\left ( q_k(s)+ir_k(s) \right )\ds \\&= \frac 1{2 \pi} \int_{r_0}^\infty \int_0^{2\pi} s^{-|k|+1}\overline{\Phi'^{-1}(z)}w(s,\varphi)e^{-ik\varphi}\ds d\varphi \\ &=\frac 1{2 \pi} \int_{B_{r_0}} \overline{\Phi'^{-1}(z)} \frac {w(\bx)}{z^k} \dx = 2ik v_{\infty,r,k}, 
\end{align*}
where $z=x_1+ix_2$.

Following by the same way as in Theorem \ref{invthm} we define affine subspace $M$ via  vorticity moments which must be invariant under fluid dynamics flows:
\begin{align}\label{noslipcondintegral2_omega} 
M=\left \{ w(\bx)\in L_1(B_{r_0})~\Big |~ 
\int_{B_{r_0}}\overline{\Phi'^{-1}(z)} \frac {w(\bx)}{z^k} \dx =  
4 \pi ik v_{\infty,r,k},~k \geq 0  \right \}.
\end{align} 
In view of (\ref{fouriercoeffr}), (\ref{fouriercoeffphi}) all moments in definition of $M$ must be equal to zero except $k=1$. 

\begin{thm}[Biot-Savart Law in exterior domain]\label{Biotpolarnoslip2}
If $w(\Phi^{-1}(z)) \in M$ then there exists the unique solution of (\ref{freediv}) - (\ref{boundinf}), (\ref{zerocirculation}) given by (\ref{BiotSavar1_2}), (\ref{BiotSavar2_2}) with Fourier coefficients $v_{r,k}$, $v_{\varphi,k} \in L_\infty(r_0,\infty)$.
\end{thm}

\begin{thm} \label{invthm}
Given initial datum $\bv_0(\bx)$ satisfying no-slip condition (\ref{bound}), infinity condition (\ref{boundinf}), zero-circularity (\ref{zerocirculation}), such that $w_0=\curl \bv_0(\Phi^{-1}(z))$ $\in L_1(\Omega)$, and $w(t,\cdot)=\curl \bv(\Phi^{-1}(z))$ be the vorticity flow. Then $w_0 \in M$ and in order to conserve no-slip condition the affine subspace $M$ must be invariant under the flow, e.g. for any time $t>0$ $w(t,\cdot) \in M$.
\end{thm}

\subsection{Stokes system}Here we prove that Robin-type boundary (\ref{robin_bound}) is really a no-slip boundary condition in exterior domains. It keeps moment relations (\ref{noslipcondintegralrieman}) under Stokes flow, and thus $M$ is invariant affine subspace and no-slip law at the boundary is satisfied.

We impose additional requirement on $\Omega$ that $\Phi^{-1}$ can be represented by absolutely converged series
\begin{equation} \label{phireq}
\Phi^{-1}(z) = z+\sum\limits_{n=1}^\infty \frac {b_n}{z^n}.
\end{equation} 

\begin{thm}\label{stokesomegathm}Let $\Omega=\RM^2 \setminus B$, where $B$ is bounded simple-connected domain with piecewise smooth boundary and $\Phi$ be a Riemann map from $\Omega$ into exterior of the disc satisfying (\ref{phireq}), given initial datum $\bv_0(\bx)$, $\curl \bv_0\in L_1(B_{r_0})$ satisfying no-slip condition (\ref{bound}), infinity condition (\ref{boundinf}), zero-circularity (\ref{zerocirculation}) and $\bv(t,\bx)$ be the solution of (\ref{maineqns})-(\ref{boundinfns}) in $\Omega$. Then $w(t,\bx)$=$\curl \bv(t,\Phi^{-1}(\bx))$ satisfies (\ref{robin_bound}).
\end{thm}

\begin{proof}
After Riemann mapping Stokes equations reduce to scalar equation on vorticity
$$
|(\Phi^{-1})'(z)|^2\partial_t w(t,x) - \Delta w=0.
$$

Fix $k>0$ and divide this equation by $z^k$ and integrate over exterior of the disc $B_{r_0}$. From moment relations (\ref{noslipcondintegral2_omega})
$$
\frac d{dt}\int_{B_{r_0}} \frac {\overline{(\Phi^{-1})'(z)}}{z^k} w(t,\bx) \dx = 0
$$ 

Using $|(\Phi^{-1})'(z)|^2=(\Phi^{-1})'(z)\overline{(\Phi^{-1})'(z)}$ and
$$
(\Phi^{-1})'(z) = \sum\limits_{n=0}^\infty \frac {c_n}{z^n}
$$
with some coefficients $c_n$ we have
\begin{align*}
\frac d{dt}\int_{B_{r_0}} |(\Phi^{-1})'(z)|^2 w(t,x) \dx = \frac d{dt}\int_{B_{r_0}}  \sum\limits_{n=0}^\infty \frac {c_n \overline{(\Phi^{-1})'(z)}}{z^{n+k}}w(t,\bx) \dx\\=\sum\limits_{n=0}^\infty c_n \frac d{dt}\int_{B_{r_0}}   \frac { \overline{(\Phi^{-1})'(z)}}{z^{n+k}}w(t,\bx) \dx=0,
\end{align*}
and thus
$$
\int_{B_{r_0}}  \frac {\Delta w}{z^{k}}\dx=0.
$$
In other hand
\begin{align*}
\int_{B_{r_0}} \frac {\Delta w}{z^{k}}\dx = \int_{r_0}^\infty \int_0^{2\pi} \frac {\Delta w}{s^ke^{ik\phi}}sdsd\varphi = 2\pi \int_{r_0}^\infty s^{-k+1}\Delta_k w_k(s) ds \\= -2\pi r_0^{-k}\left(r_0 \frac{\partial w_k(t,r)}{\partial r}\Big|_{r=r_0} + k w_k(t,r_0) \right )=0.
\end{align*}

Using $w_{-k}(t,r) = \overline{w_k(t,r)}$ we will have condition (\ref{robin_bound}) for $k<0$.

\end{proof}

\subsection{Navier-Stokes system}

Navier-Stokes equation (\ref{maineqw}) after Riemann mapping reduces to 
$$
|(\Phi^{-1})'(z)|^2\partial_t w(t,x) - \Delta w + B(v,w) =0,
$$
where
$$
B(v,w) = \real (\Phi^{-1})' (\bv,\nabla w) - \imag (\Phi^{-1})' (\bv^\perp,\nabla w).
$$

Supply this equation with

\begin{align}
&\left(r_0 \frac{\partial w_k(t,r)}{\partial r}\Big|_{r=r_0} +  |k| w_k(t,r_0) \right ) = \nonumber \\
&~~~~\left \{ {  \begin{matrix}
\frac {|k| r_0^{|k|}}{2\pi}
\int_{B_{r_0}} (\bv^\CM_\infty - \bv^\CM){z^{-|k|-1}} \overline{(\Phi^{-1})'} w(t,\bx)\dx,~k\geq 0, \\
\frac {|k| r_0^{|k|}}{2\pi}
\int_{B_{r_0}} \overline{(\bv^\CM_\infty - \bv^\CM)z^{-|k|-1}} (\Phi^{-1})' w(t,\bx)\dx,~k<0.
\end{matrix} } \right . \label{robin_bound_ns_general}
\end{align}


\begin{thm}Let $\Omega$, $\Phi$ as in Theorem \ref{stokesomegathm}, given initial datum $\bv_0(\bx)$, $\curl \bv_0\in L_1(\Omega)$ satisfying no-slip condition (\ref{bound}), infinity condition (\ref{boundinf}), zero-circularity (\ref{zerocirculation}) and $\bv(t,\bx)$ be the solution of (\ref{maineqns})-(\ref{boundinfns}) in $\Omega$. Then $w(t,\cdot)=\curl \bv(t,\Phi^{-1}(z))$ satisfies (\ref{robin_bound_ns_general}).
\end{thm}

\begin{proof}
Apply the same steps as in previous subsections: for fixed $k>0$ divide this equation by $z^k$ and integrate over $B_{r_0}$. In previous subsection we obtained
$$
\frac d{dt}\int_{B_{r_0}} |(\Phi^{-1})'(z)|^2 w(t,x) \dx = 0,
$$
and
\begin{align*}
\int_{B_{r_0}} \frac {\Delta w}{z^{k}}\dx = -2\pi r_0^{-k}\left(r_0 \frac{\partial w_k(t,r)}{\partial r}\Big|_{r=r_0} + k w_k(t,r_0) \right ).
\end{align*}

From (\ref{freediv3}), (\ref{curleq3})
$$
 \int_{B_{r_0}} \frac {\real (\Phi^{-1})'}{z^{k}} w \divv \bv \dx  = -
 \int_{B_{r_0}} \frac {\imag (\Phi^{-1})'}{z^{k}} w \curl \bv   \dx.  
$$

Then using Cauchy-Riemann equations
\begin{align*}
\int_{B_{r_0}} \frac {B(v,w)}{z^{k}}\dx = \int_{B_{r_0}} \frac {\real (\Phi^{-1})'}{z^{k}}(\bv,\nabla w) -
 \frac {\imag (\Phi^{-1})'}{z^{k}} (\bv^\perp,\nabla w) \dx\\=
 \int_{B_{r_0}} \left ( -\frac {\real (\Phi^{-1})'}{z^{k}} w \divv \bv  -
 \frac {\imag (\Phi^{-1})'}{z^{k}} w \curl \bv \right )  \dx  \\
 -\int_{B_{r_0}} w \left ( \left ( v, \nabla \frac {\real (\Phi^{-1})'}{z^{k}} \right ) - \left ( v^\perp, \nabla \frac {\imag (\Phi^{-1})'}{z^{k}} \right ) \right ) \dx \\ =
 k\int_{B_{r_0}} \frac {\overline{(\Phi^{-1})'}}{z^{k+1}} (v_1+iv_2) w(t,\bx) \dx
 =k\int_{B_{r_0}} \frac {\overline{(\Phi^{-1})'}}{z^{k+1}} \bv^\CM w(t,\bx)   \dx
\end{align*}

Since from (\ref{noslipcondintegral2_omega})
$$
k\int_{B_{r_0}} \frac {\overline{(\Phi^{-1})'}}{z^{k+1}} \bv^\CM_\infty w(t,\bx)   \dx = 0, 
$$
then
$$
\int_{B_{r_0}} \frac {B(v,w)}{z^{k}}\dx = k\int_{B_{r_0}} \frac {\overline{(\Phi^{-1})'}}{z^{k+1}} (\bv^\CM-\bv^\CM_\infty ) w(t,\bx)   \dx,
$$
from which follows (\ref{robin_bound_ns_general}) for $k\geq 0$.

For $k<0$ (\ref{robin_bound_ns_general}) follows from the identity $w_{-k}(t,r) = \overline{w_k(t,r)}$. Theorem is proved.

\end{proof}

This theorem says, that in case of Oseen approximation we can use Robin-type boundary (\ref{robin_bound}) in general domain $\Omega$. And for Navier-Stokes system if $\|\bv^\CM-\bv^\CM_\infty\|$ is small, then (\ref{robin_bound}) works well. Nevertheless this boundary condition requires some corrections such as boundary control (\ref{robin_bound_control}) in order to stay on invariance affine subspace $M$.


\begin{thebibliography}{99}

\bibitem{A}Anderson, C.  Vorticity Boundary Conditions and Boundary Vorticity Generation for Two Dimensional Viscous Incompressible Flows, J. Comp. Phys. 80 (1989).

\bibitem{B}Batchelor, G. K. An Introduction to Fluid Dynamics, Cambridge University Press (1967), ISBN 0-521-09817-3.

\bibitem{AG} Gorshkov, A. Associated Weber–Orr Transform, Biot–Savart Law and Explicit Form of the Solution of 2D Stokes System in Exterior of the Disc. J. Math. Fluid Mech. 21, 41 (2019). https://doi.org/10.1007/s00021-019-0445-2

\bibitem{BE} Bateman, H.,  Erdelyi, A. Higher Transcendental Functions, Vol. II. McGraw-Hill, New York (1953).




\end{thebibliography}
\end{document}